\documentclass[11pt]{amsart}
\usepackage{pxfonts}
\usepackage{verbatim}
\usepackage{amsmath,amsfonts,amscd,amssymb,epsfig,euscript,bm}
\usepackage{amsxtra,mathrsfs}
\usepackage{pdfsync}
\newtheorem{theorem}{Theorem}[section]

\newtheorem{lemma}{Lemma}[section]

\theoremstyle{definition}
{}

\theoremstyle{remark} 
\newtheorem{remark}{Remark}[section]

\newcommand{\pbp}{\mathcal{\partial \bar{\partial}}}
\newcommand{\spbp}{\mathcal{\sqrt{-1}\partial \bar{\partial}}}
\newcommand{\ddc}{\frac{\sqrt{-1}}{2\pi}\mathcal{\partial \bar{\partial}}} 

\numberwithin{equation}{section}
\usepackage{hyperref}
\hypersetup{colorlinks,citecolor=blue,plainpages=false,hypertexnames=false}
\begin{document}
\title[Representability of Chern-Weil forms]{Representability of Chern-Weil forms}
\author{Vamsi Pritham Pingali}
\address{Department of Mathematics, Indian Institute of Science, Bangalore, India - 560012}
\email{vamsipingali@math.iisc.ernet.in}
\begin{abstract}
In this paper we look at two naturally occurring situations where the following question arises. When one can find a metric so that a Chern-Weil form can be represented by a given form ? The first setting is semi-stable Hartshorne-ample vector bundles on complex surfaces where we provide evidence for a conjecture of Griffiths by producing metrics whose Chern forms are positive. The second scenario deals with a particular rank-2 bundle (related to the vortex equations) over a product of a Riemann surface and the sphere.
\end{abstract}
\maketitle
\section{Introduction}
\indent In \cite{Pin} the author introduced a geometric question : when one can find a metric so that a Chern-Weil form can be represented by a given form in the correct cohomology class ? Apart from a few very special cases (all of which involve conformally changing a given metric to another one) the answer is not known in general. \\
\indent In this paper we study this question in the context of two natural situations in algebraic geometry and mathematical physics. The first of these is related to a conjecture of Griffiths dealing with Hartshorne-ample vector bundles. A holomorphic vector bundle $E$ over a compact complex manifold $X$ is said to be Hartshorne-ample if the canonical line bundle $\mathcal{O}_E(1)$ over $\mathbb{P}(E)$ is ample. If $E$ admits a metric whose curvature $\Theta$ is Griffiths-positive, i.e., $\langle v, \Theta v\rangle$ is a positive $(1,1)$-form for all vectors $v \in E$, then it is easy to see that it is Hartshorne-ample. The converse is the Griffiths conjecture. In the case of curves this was proven in \cite{Um,Cam}. Somewhat strong evidence for this conjecture in the general case is provided by the fact that if $E$ is Hartshorne-ample, then $E\otimes \det(E)$ is Nakano-positive (stronger than Griffiths positive) \cite{Dem, Mou}. It is also known \cite{BG, FL} that the Schur polynomials of Hartshorne-ample bundles are numerically positive (and in fact the only numerically positive characteristic classes are positive linear combinations of the Schur polynomials \cite{FL}). In the case of bundles over surfaces, this means that $c_1^2, c_2, c_1, c_1^2-c_2$ are numerically positive (and generate the ``numerically positive cone"). It is but natural to ask whether there is a metric on $E$ such that the corresponding Chern-Weil forms are positive. The following theorem addresses this question in the case of compact complex surfaces.
\begin{theorem}
Let $X$ be a compact complex surface and $E$ be a Hartshorne-ample rank-$r$ holomorphic vector bundle on it. Assume that $E$ is semi-stable with respect to some polarisation $[L]$. Then there exists a hermitian metric $h$ on $E$ whose Chern-Weil forms satisfy $c_1 (h)>0, c_2(h)>0, c_1^2(h)- c_2(h)>0$.
\label{stable}
\end{theorem}
\begin{remark}
It is worth noting that the assumption of semi-stability is actually quite natural in this context. The reason is that Umemura's proof \cite{Um} in the case of curves uses the concept of stability.
\label{firstrem}
\end{remark}
\indent The second question we deal with arises from the gravitating vortex equations, which are themselves a special case of the K\"ahler-Yang-Mills equations studied in \cite{Garcia1, Garcia2, Garcia3}. The K\"ahler-Yang-Mills equations for a metric $H$ on a bundle $E$ and a K\"ahler metric $\omega$ on a manifold $X$ are
\begin{gather}
\sqrt{-1} \Theta \wedge \omega^{n-1} = \lambda \omega^n \ Id \nonumber \\
S_{\omega} - c = \alpha \frac{\mathrm{ch}_2 (E,H) \wedge \omega^{n-2}}{\omega^n}
\label{origKYM}
\end{gather}
where $c$ is a constant, $S_{\omega}$ the scalar curvature of $\omega$, $\Theta$ is the curvature of $H$, and $ch_2(E,H)$ is the second Chern character form. These equations admit a moment map interpretation just like the usual Hermite-Einstein equation and the constant scalar curvature K\"ahler (cscK) equation \cite{Garcia1}. Special cases of these equations have been studied. In particular, in \cite{Garcia2} a perturbation result (theorem 1.1) around $\alpha=0$ was proven for $\Sigma\times\mathbb{P}^1$ (where $\Sigma$ is a genus $\geq 1$ Riemann surface) equipped with a certain $SU(2)$-invariant rank $2$-vector bundle (that we dub ``the vortex bundle"). For $\mathbb{P}^1\times \mathbb{P}^1$ an obstruction was found (akin to K-stability). Solving this equation in general is obviously quite challenging. However, taking cue from the Calabi volume conjecture which is easier than the problem of cscK metrics, intersects with cscK metrics in the Ricci flat case,  and has no obstruction, we propose to study the following Calabi-Yang-Mills equations. 
\begin{gather}
\sqrt{-1} \Theta \wedge \omega^{n-1} = \lambda \omega^n \ Id \nonumber \\
\omega^n - \eta = \alpha \mathrm{ch}_2 (E,H) \wedge \omega^{n-2}
\label{CYMorig}
\end{gather}
where $\eta>0$ is a given volume form in the right cohomology class. Note that if $\alpha=0$ this is simply the usual Calabi conjecture. In this paper, the following existence result is proven for these equations.
\begin{theorem}
Let $\Sigma$ be a Riemann surface and $X= \Sigma \times \mathbb{P}^1$. Let $SU(2)$ act trivially on $X$ and in the standard manner on $\mathbb{P}^1=SU(2)/U(1)$. Let $L$ be a holomorphic line bundle on $\Sigma$ and $E$ be a rank $2$ holomorphic bundle over $X$ which is an extension :
\begin{gather}
0 \rightarrow \pi_1^{*} L\rightarrow E \rightarrow \pi_2^{*}\mathcal{O}(2) \rightarrow 0. 
\label{ext}
\end{gather}
Let $\tau>0$ be a constant, $\omega_{FS}= \frac{idz\wedge d\bar{z}}{(1+\vert z\vert^2)^2}$ the Fubini-Study metric on $\mathbb{P}^1$, and $\omega_{\Sigma}$ a metric on $\Sigma$. Denote by $\Omega=\pi_1 ^{*} \omega_{\Sigma}+\frac{4}{\tau} \pi_2 ^{*} \omega_{FS}$ an $SU(2)$-invariant K\"ahler form on $X$ where $\displaystyle \int _{\Sigma} \omega_{\Sigma} = \mathrm{Vol}(\Sigma)$ is fixed, by $H$ an $SU(2)$-invariant hermitian metric on $E$, and by $\Theta$ the curvature of $H$. Also let $\alpha$ be a constant, $ch_2(E,H) = \frac{1}{2} \mathrm{tr}\left(\frac{\sqrt{-1} \Theta}{2\pi}\right )^2$ be the second Chern character form, and $\eta$ be an $SU(2)$-invariant $(2,2)$-form on $X$.
The following statements hold. 
\begin{enumerate}
\item For any  $\Omega>0$ and $\eta>0$ satisfying  $\displaystyle \int _X \Omega^2 = \int _X \eta$, there exists an $H$ such that $\Omega^2 + \alpha ch_2(E,H) = \eta$. 
\item  Assume that $0<c_1(L)<\frac{\tau \mathrm{Vol}(\Sigma)}{4\pi}$. Suppose we are given an $\eta>0$ satisfying  $\displaystyle \int _X \Omega^2 = \int _X \eta$. Then the set of $\alpha \geq 0$ satisfying 
\begin{gather}
8+\frac{2\alpha\tau}{(2\pi)^2} \left [2\lambda - \frac{\tau}{2} \right ]>0
\label{satisfaction}
\end{gather}
 for which there exists a smooth form $\Omega_{\alpha}>0$ and a smooth metric $H_{\alpha}$ such that the Calabi-Yang-Mills equations are satisfied, i.e.,
\begin{gather}
\sqrt{-1}\Theta_{\alpha}\wedge \Omega_{\alpha} = \lambda \Omega_{\alpha} ^2  Id\nonumber \\
\Omega_{\alpha}^2 + \alpha ch_2(E,H_{\alpha}) = \eta, 
\label{pKYM}
\end{gather}
contains $\alpha=0$ and is open. Moreover, for $\alpha=0$ the solution is essentially unique among all $SU(2)$-invariant solutions.
\end{enumerate} 
\label{KYM}
\end{theorem}
Notice that unlike the case of the K\"ahler-Yang-Mills equations, there seems to be no difference between the case of $g=0$ and $g\geq 1$. Moreover, the result is more general than the corresponding one in \cite{Garcia2} for the K\"ahler-Yang-Mills equations. We plan on  exploring these equations further in future work.\\

\emph{Acknowledgements} : The author is grateful to Harish Seshadri for useful suggestions, as well as for his support and encouragement. We also thank M.S. Narasimhan for pointing out the existence of approximate Hermite-Einstein metrics on semi-stable bundles, and Mario Garcia-Fernandez for answering questions about his paper. Lastly, our gratitude extends copiously to the anonymous referee for a careful reading of this work.
\section{Semi-stable ample bundles on surfaces}
In this section we study stable Hartshorne-ample bundles on compact complex surfaces. We need to use two known facts.
\begin{enumerate}
\item Let $E$ be a holomorphic vector bundle of rank $r$ over a compact complex surface $X$. Let $L=\mathcal{O}_E(1)$ be the canonical bundle over $\mathbb{P}(E)$. At the level of classes it is well known that $\pi_{*}([c_1(L)^{r}]) = [c_1(E)]$ and $\pi_{*}([c_1 ^{r+1} (L)]) = [c_1^2(E)-c_2(E)]$ where $\pi_{*}$ is the fibre-integral. As a consequence (see lemma $4.1.1$ from \cite{Gulerelec} for instance) it follows that $c_1(E)$ and $c_1^2(E)-c_2(E)$ have positive representatives\footnote{It follows from a theorem proven in \cite{Guler, Div} that this equality holds even at the level of Chern-Weil forms for Griffiths-positive bundles.}. In particular, $X$ is projective by the Kodaira embedding theorem.  
\item A theorem of Fulton-Lazarsfeld \cite{FL} (building on the work of Bloch-Gieseker \cite{BG}) shows that for a Hartshorne-ample bundle over a projective surface, $c_1, c_1^2, c_2$ and $ c_1^2-c_2$ are all numerically positive, i.e., when integrated over subvarieties over appropriate dimensions one gets positive numbers. Actually this theorem holds true in greater generality (the only numerically positive classes are positive linear combinations of Schur polynomials of Chern classes). But for our purposes this statement is good enough. 
\end{enumerate}
\emph{Proof of theorem \ref{stable}} :  \\
\indent Let $\omega$ be a metric in the class $c_1(L)$. The assumption of semi-stability with respect to the polarisation $L$ and a theorem of Kobayashi (theorem $10.13$ in \cite{Kob}) shows that the rank-$r$ bundle $E$ is approximately Hermite-Einstein, i.e., for every given $1>\epsilon>0$, there exists a metric $h_{0, \epsilon}$ satisfying
\begin{gather}
\displaystyle \Vert \frac{\sqrt{-1} \Theta_{0, \epsilon} \wedge \omega}{ \omega^2}-\lambda Id \Vert < \epsilon,
\label{HE}
\end{gather}
where $\Theta_{0,\epsilon}$ is the curvature of the Chern connection of $h_{0,\epsilon}$, and $\lambda =\frac{\displaystyle \int _X c_1(E) \omega}{r\displaystyle \int_X \omega^2}$ is a positive constant. For this metric the proof of the Kobayashi-L\"ubke inequality  (theorem $5.7$ in \cite{Kob}) shows that
\begin{gather}
(r-1)c_1(E,h_{0,\epsilon})^2 -2r c_2 (E,h_{0,\epsilon}) \leq C \epsilon 
\omega^2,
\label{KL}
\end{gather}
where $C$ is a constant depending only on $\lambda, r$. \\
\indent We supress the dependence on $\epsilon$ from now. It is not clear that the first Chern form satisfies $c_1(E,h_{0,\epsilon})>0$. If it did, then $c_2$ would be positive as well. We conformally change the metric $h=h_0 e^{-\phi}$ in the hope that for appropriately chosen $\phi$ this new metric satisfies the conditions of the theorem. We compute the new Chern-Weil forms : 
\begin{gather}
\Theta _h = \Theta _0 + \pbp \phi Id \nonumber \\
c_1 (h) =c_1(h_0) + r \ddc \phi \nonumber \\
c_2 (h) = c_2 (h_0) + (r-1) \ddc \phi \wedge c_1 (h_0) + \frac{r(r-1)}{2} \left(\ddc \phi \right)^2.
\label{newchern}
\end{gather}
Using equations \ref{newchern} we compute the following linear combination of Chern forms (the second Segre form). 
\begin{gather}
c_1 ^2 (h) - c_2(h) = \left(c_1(h_0) + r \ddc \phi \right) ^2 \nonumber \\  -   \left(c_2 (h_0) + (r-1) \ddc \phi \wedge c_1 (h_0) + \frac{r(r-1)}{2} \left(\ddc \phi \right)^2\right ) \nonumber \\
= c_1 ^2 (h_0) - c_2(h_0) + (r+1) \ddc \phi \wedge c_1 (h_0) +\frac{r(r+1)}{2}\left(\ddc \phi \right)^2\nonumber \\
= c_1 ^2 (h_0) - c_2(h_0) +\frac{r(r+1)}{2} \left (\ddc \phi + \frac{c_1 (h_0)}{r} \right )^2 - \frac{(r+1) c_1(h_0)^2 }{2r} \nonumber \\
= \frac{(r-1)c_1^2(h_0)-2r c_2(h_0)}{2r} +\frac{r(r+1)}{2} \left (\ddc \phi + \frac{c_1 (h_0)}{r} \right )^2 .
\label{lincom}
\end{gather}
As mentioned earlier, we know that $c_1 ^2 -c_2 $ has a positive representative $\eta$. Therefore we may attempt to find a $\phi$ solving 
\begin{gather}
\frac{r(r+1)}{2} \left (\ddc \phi + \frac{c_1 (h_0)}{r} \right )^2 = \eta + \frac{2r c_2(h_0)-(r-1)c_1^2(h_0)}{2r}. 
\label{MAphi}
\end{gather}
Thanks to inequality \ref{KL} we see that for sufficiently small $\epsilon$ the right hand side is positive. The class $[c_1(h_0)]$ is a K\"ahler class. Moreover, the integrals of both sides of the equation are equal. Indeed,
\begin{gather}
\displaystyle \int _X \frac{r(r+1)}{2} \left (\ddc \phi + \frac{c_1 (h_0)}{r} \right )^2 = \int _X \frac{r(r+1)}{2}  \left ( \frac{c_1 (h_0)}{r} \right )^2 \nonumber \\
\displaystyle \int_X \left ( \eta + \frac{2r c_2(h_0)-(r-1)c_1^2(h_0)}{2r} \right ) = \int_X \left ( c_1^2-c_2 + \frac{2r c_2-(r-1)c_1^2}{2r} \right ) \nonumber \\
= \int _X \frac{r(r+1)}{2}  \left ( \frac{c_1 (h_0)}{r} \right )^2.
\end{gather}
Therefore, by Yau's solution of the Calabi conjecture \cite{Yau} we have a smooth function $\phi$ solving equation \ref{MAphi} such that $\ddc \phi + \frac{c_1 (h_0)}{r} = \frac{c_1 (h)}{r}>0$. Computing the second Chern form using equations \ref{MAphi} and \ref{newchern} we see that
\begin{gather}
c_2 (h) = c_2 (h_0) +\frac{r-1}{r+1} \left (\eta + c_2 (h_0)-c_1^2(h_0) \right ) 
= \frac{2rc_2 (h_0) - (r-1)c_1 ^2 (h_0) + (r-1) \eta}{r+1}, \nonumber
\end{gather}
which is positive for small $\epsilon$. Therefore we have successfully found a metric $h$ on the bundle $E$ satisfying
\begin{gather}
\Vert \frac{\sqrt{-1} \Theta_h \wedge \omega}{\omega^2} -(\lambda+\Delta \phi) Id \Vert < \epsilon , \ c_1(h) >0, \  c_2 (h)>0, \  and \ c_1^2 (h) - c_2 (h) >0. \nonumber
\end{gather}
\qed \\
In addition, we observe that
\begin{gather}
(r-1)c_1 ^2 (h)-2r c_2 (h) = (r-1) \left(c_1(h_0) + r \ddc \phi \right) ^2  \nonumber \\ -  2r \left(c_2 (h_0) + (r-1) \ddc \phi \wedge c_1 (h_0) + \frac{r(r-1)}{2} \left(\ddc \phi \right)^2\right ) \nonumber \\
=(r-1)c_1 ^2 (h_0)-2r c_2 (h_0) \leq C\epsilon \omega^2.
\label{inv}
\end{gather}
It is tempting to hope that given the large number of conditions being satisfied, the curvature $\Theta_h$ of the metric $h$ is Griffiths positive. But it is not clear and we suspect that it is unlikely.
\section{Chern forms of a vortex bundle}
In order to prove theorem \ref{KYM} we need to calculate the curvature of an $SU(2)$-invariant metric $H$ on $E$. We follow the calculations in \cite{Garcia1, Oscar}. Consider the metric $H=h_1 \oplus f_2 \frac{8\pi}{\tau}\frac{dz\otimes d\bar{z}}{(1+\vert z \vert^2)^2}$ on $\pi_1^{*}L \oplus \pi_2 ^{*} \mathcal{O}(2)$ where $h_1$ is a metric on $L$ over $\Sigma$, and $f_2$ is a function on $\Sigma$. For future use, define $h=\frac{h_1}{f_2}$ and $g_2 = f_2 \frac{8\pi}{\tau}\frac{dz\otimes d\bar{z}}{(1+\vert z \vert^2)^2}$. Notice that every holomorphic structure $E$ on the smooth complex bundle $\pi_1^{*}L \oplus \pi_2 ^{*} \mathcal{O}(2)$ can be given by an element $\beta \in H^1(X,\pi_1^{*}L \otimes \pi_2^{*} \mathcal{O}(2)) = H^{0}(\Sigma, L)$. Every such $\beta$ is of the form 
\begin{gather}
\beta =\pi_1^{*} \phi \otimes \pi_2^{*} \zeta, \nonumber 
\end{gather}
where $\phi \in H^0(\Sigma, L)$ and $\zeta = \frac{\sqrt{8\pi}}{\tau} \frac{dz\otimes d\bar{z}}{(1+\vert z \vert^2)^2}$. In an orthonormal frame the Chern connection on $E$ associated to $H$ and $\beta$ is of the following form. \begin{gather}
A= \left (\begin{array}{cc}A_{h_1} & \beta \\ -\beta^{*} & A_{g_2}\end{array} \right ) \nonumber.
\end{gather}
The curvature is
\begin{gather}
\Theta =  \left (\begin{array}{cc}\Theta_{h_1} - \beta \wedge \beta^{*} & \nabla ^{(1,0)}\beta \\ -\nabla^{(0,1)}\beta^{*} & \Theta_{g_2} - \beta^{*} \wedge \beta \end{array} \right ). 
\label{Curv} 
\end{gather}
We note that \cite{Oscar}
\begin{gather}
\beta \wedge \beta^{*} = \frac{\sqrt{-1}}{\tau} \vert \phi \vert_h ^2 \omega_{FS} \nonumber \\
\nabla ^{(1,0)}\beta \wedge \nabla^{(0,1)}\beta^{*} = - \frac{\sqrt{-1}}{\tau} \nabla^{(1,0)} \phi \wedge \nabla ^{(0,1)}\phi^{*} \wedge \omega_{FS},
\label{oscaruseful}
\end{gather}
where as before, $h=\frac{h_1}{f_2}$. Upon calculation (in normal coordinates) we see that
\begin{gather}
\pbp \vert \phi \vert_h ^2 = -\Theta _h \vert \phi \vert_h ^2 + \nabla^{(1,0)} \phi \wedge \nabla ^{(0,1)}\phi^{*}.
\label{normalcood}
\end{gather}
Now we compute the terms in equation \ref{pKYM}.
\begin{gather}
\sqrt{-1} \Theta  \Omega -\lambda \Omega^2 \ Id \nonumber \\ =  \left ( \begin{array}{cc} \sqrt{-1} \Theta_{h_1} \frac{4}{\tau}\omega_{FS} +\frac{\vert \phi \vert_h ^2}{\tau} \omega_{FS}\omega_{\Sigma} - 2\lambda \frac{4}{\tau} \omega_{\Sigma} \omega_{FS}& 0 \\ 0 & (\sqrt{-1}\Theta_{f_2}+2\omega_{FS}) \Omega -\frac{\vert \phi \vert_h ^2}{\tau} \omega_{FS}\omega_{\Sigma} - 2\lambda \frac{4}{\tau} \omega_{\Sigma} \omega_{FS}\end{array} \right ) \nonumber \\
= \frac{4}{\tau}\omega_{FS} \left ( \begin{array}{cc}  \sqrt{-1}(\Theta_{h}+\Theta_{f_2})  +(\frac{\vert \phi \vert_h ^2}{4} -2\lambda)\omega_{\Sigma}& 0 \\ 0 & \sqrt{-1}\Theta_{f_2}+(\frac{\tau}{2} -\frac{\vert \phi \vert_h ^2}{4}  - 2\lambda)  \omega_{\Sigma} \end{array} \right ) 
\label{termsinpKYMone}
\end{gather}
\begin{gather}
\mathrm{ch}_2(E,H) = -\frac{1}{2(2\pi)^2} \mathrm{Tr}(\Theta^2) \nonumber \\
= -\frac{1}{2(2\pi)^2} \left ((\Theta_{h_1}-\beta\wedge \beta^{*})^2-2\nabla ^{(1,0)}\beta \wedge \nabla^{(0,1)}\beta^{*}+(\Theta_{g_2} - \beta^{*} \wedge \beta)^2 \right ) \nonumber \\
=\frac{\sqrt{-1}}{\tau(2\pi)^2} \left (\Theta_{h_1}\vert \phi \vert _h ^2 \omega_{FS} -\nabla^{(1,0)} \phi \wedge \nabla ^{(0,1)}\phi^{*} \wedge \omega_{FS}+2\tau\Theta_{f_2} \omega_{FS} -\Theta_{f_2}\vert \phi \vert _h ^2 \omega_{FS}\right )\nonumber \\
=\frac{\sqrt{-1}}{\tau(2\pi)^2}\omega_{FS} \left (\Theta_{h}\vert \phi \vert _h ^2  -\nabla^{(1,0)} \phi \wedge \nabla ^{(0,1)}\phi^{*} +2\tau\Theta_{f_2} \right )
= \frac{\sqrt{-1}}{\tau(2\pi)^2}\omega_{FS} \left (-\pbp \vert \phi \vert_h^2 +2\tau\Theta_{f_2} \right ).
\label{termsinpKYMtwo}
\end{gather}
Since $\eta$ is $SU(2)$-invariant, it is of the form $\eta =\frac{8}{\tau} f \wedge \omega_{FS}$ where $f$ is the pull-back of a volume form from $\Sigma$. Therefore we calculate the second equation in the Calabi-Yang-Mills equations as follows.
\begin{gather}
\Omega^2 + \alpha \mathrm{ch}_2 (E,H) - \eta = \omega_{FS} \left(\frac{8}{\tau}\omega_{\Sigma}+\frac{\sqrt{-1}\alpha}{\tau(2\pi)^2} (-\pbp \vert \phi \vert_h^2 +2\tau\Theta_{f_2}) -\frac{8}{\tau}f \right ).
\label{termsinpKYMtwocontd}
\end{gather} 
For the record we note that (upon integration on both sides) $\lambda = \frac{\tau}{8} + \frac{c_1(L)\pi}{2\mathrm{Vol}(\Sigma)}$. We now proceed to prove theorem \ref{KYM}. \\

\emph{Proof of part 1) of theorem \ref{KYM}} : Since $\int \Omega^2 = \int \eta$ it is clear that $\frac{8}{\tau}\omega_{\Sigma} - \frac{8}{\tau}f = \frac{\alpha}{\tau (2\pi)^2}\spbp u$ for some function $u$. Using equation \ref{termsinpKYMtwocontd} we see that the right-hand side is zero if and only if $\vert \phi \vert_h^2 +2\tau \ln (f_2) = u$. Therefore we can easily choose (an infinite number of) $f_2$ satisfying the equation in the first part of theorem \ref{KYM}.
\qed
\\

\emph{Proof of part 2) of theorem \ref{KYM}} : The Calabi-Yang-Mills equations are (using \ref{termsinpKYMone}, \ref{termsinpKYMtwocontd}) 
\begin{gather}
\sqrt{-1}(\Theta_{h}+\Theta_{f_2})  +(\frac{\vert \phi \vert_h ^2}{4} -2\lambda)\omega_{\Sigma} = 0 \nonumber \\
 \sqrt{-1}\Theta_{f_2}+(\frac{\tau}{2} -\frac{\vert \phi \vert_h ^2}{4}  - 2\lambda)  \omega_{\Sigma} = 0 \nonumber \\
8\omega_{\Sigma}+\frac{\sqrt{-1}\alpha}{(2\pi)^2} (-\pbp \vert \phi \vert_h^2 +2\tau\Theta_{f_2}) -8f = 0.
\label{CYM}
\end{gather}
\emph{The case of $\alpha=0$} : When $\alpha=0$, the third equation can be solved trivially for a unique, smooth $\omega_{\Sigma}$. The other two can also be solved for a smooth solution because $E$ is stable \cite{Oscar} and therefore Donaldson's existence theorem \cite{donald} applies. As for uniqueness, from the second equation it is clear that a unique $h$ ensures a unique $f_2$ with zero-average. Substituting the second equation in the first, we get the following equation. 
\begin{gather}
\sqrt{-1}\Theta_h +\left ( \frac{\vert \phi \vert_h ^2 -\tau}{2} \right )\omega_{\Sigma} = 0. 
\label{uni}
\end{gather}
Suppose there are two solutions $\mathbf{h}_1 =\mathbf{h}_0 e^{-\psi_1}$ and $\mathbf{h}_2 = \mathbf{h}_0 e^{-\psi_2}$ where $\mathbf{h}_0$ is a metric on $L$, then upon subtraction we get
\begin{gather}
\spbp (\psi_2-\psi_1) + \frac{\vert \phi \vert_{\mathbf{h}_2} ^2 }{2}\left ( 1-e^{\psi_2-\psi_1} \right ) = 0. 
\label{unicontd}
\end{gather}
At the global maximum of $\psi_2-\psi_1$ we know that $\spbp (\psi_2-\psi_1) \leq 0$. Therefore $1\geq e^{\psi_2-\psi_1}$. This means that $\psi_2 \leq \psi_1$ throughout. Interchanging the roles of $\psi_1$ and $\psi_2$ we see that $\psi_1=\psi_2$. \footnote{Incidentally, equation \ref{unicontd} can be solved using the existence theorem of Kazdan-Warner \cite{KW}. This provides an alternative proof of existence.}\\
\indent Before we proceed to prove openness along this ``continuity path", we note that $\omega_{\Sigma}$ is always positive along this path. Indeed, if $\omega_{\Sigma} (p)=0$ at some point $p$ for the first such value of $\alpha>0$, then the first two equations of \ref{CYM} show that $\Theta_{h_1}=\Theta_{f_2}=0$ at $p$. Choosing normal coordinates at $p$ we see that the third equation implies
\begin{gather}
-\frac{\sqrt{-1}\alpha}{(2\pi)^2} \partial \phi(p) \wedge  \bar{\partial} \bar{\phi} (p) = 8f(p) > 0.
\end{gather}
This is a contradiction. Therefore, as long (measured in $\alpha$) as a solution exists, $\omega_{\Sigma}>0$.  \qed\\

\emph{Openness} : Now we prove openness.  Firstly, we claim that
\begin{gather} 
\vert \phi \vert_h ^2 - \tau \leq 0. 
\label{usefineq}
\end{gather}
Indeed, let the maximum of $\vert \phi \vert_h ^2$ be attained at a point $p$. Then $$0 \geq \spbp \vert \phi \vert_h ^2 (p)  \geq -\Theta_h (p) \vert \phi \vert_h ^2 (p).$$ Using equation \ref{uni} we see that $\vert \phi \vert_h ^2 (p) \leq \tau$. Now consider the following Banach spaces/manifolds of functions on $\Sigma$ equipped with the fixed background metric $f$. $\mathcal{B}_0 ^{k+2,a}$ is the subspace of $C^{k+2,a}$ H\"older functions with zero average, $\mathcal{B}_{>0} ^{k+2,a}$ is the open subset of $\mathcal{B}_0 ^{k+2,a}$ with elements $\varphi$ satisfying $\omega_{\Sigma}=f+\spbp \varphi >0$, and $\mathcal{B}_{sub} ^{k+2,a}$ is the submanifold of $\mathcal{B} ^{k+2,a} \times \mathcal{B}_{>0} ^{k+2,a}$ consisting of pairs of functions $(\psi, \varphi)$ such that $\displaystyle \int_{\Sigma} \vert \phi \vert_{h_0} ^2 e^{-\psi}(f+\spbp \varphi) = \mathrm{Vol}(\Sigma) (2\tau-8\lambda)$. Note that the tangent space $T\mathcal{B}_{sub} ^{k+2,a}$ of $\mathcal{B}_{sub} ^{k+2,a}$ at $(\psi,\varphi)$ consists of functions $\dot{\psi}, \dot{\varphi}$ such that  
\begin{gather}
\displaystyle \int \dot{\varphi} f =0 \nonumber \\
\displaystyle \int \left (-\vert \phi \vert _{h_0}^2 e^{-\psi} \dot{\psi} (f+\spbp \varphi) + \vert \phi \vert _{h_0}^2 e^{-\psi} \spbp \dot{\varphi} \right)=0.
\label{tangent}
 \end{gather}
Define the map $T:B_1  =\mathbb{R} \times \mathcal{B}_{sub} ^{k+2,a} \times \mathcal{B}_0 ^{k+2,a} \rightarrow B_2 = \mathcal{B}_0 ^{k,a} \times \mathcal{B}_0 ^{k,a} \times \mathcal{B}_0 ^{k,a}$ as  
\begin{gather}
T(\alpha, \psi, \varphi,  \psi_2) = (T_1,T_2,T_3) \ where \ \nonumber \\
T_1 = \sqrt{-1}(\Theta_{h_0 e^{-\psi}}+\Theta_{e^{-\psi_2}})  +(\frac{\vert \phi \vert_{h_0}^2 e^{-\psi} }{4} -2\lambda)(f+\spbp \varphi)  \nonumber \\
=\sqrt{-1}(\Theta_{h_0}+\pbp \psi+\pbp \psi_2)  +(\frac{\vert \phi \vert_{h_0}^2 e^{-\psi} }{4} -2\lambda)(f+\spbp \varphi) \nonumber \\
T_2 = \sqrt{-1}\Theta_{e^{-\psi_2}}+(\frac{\tau}{2} -\frac{\vert \phi \vert_{h_0} ^2e^{-\psi}}{4}  - 2\lambda)  (f+\spbp \varphi) \nonumber \\
= \spbp\psi_2+(\frac{\tau}{2} -\frac{\vert \phi \vert_{h_0}^2e^{-\psi}}{4}  - 2\lambda)  (f+\spbp \varphi) \nonumber \\
T_3 = 8(f+\spbp \varphi)+\frac{\sqrt{-1}\alpha}{(2\pi)^2} (-\pbp (\vert \phi \vert_{h_0}^2 e^{-\psi}) +2\tau\Theta_{e^{-\psi_2}}) -8f \nonumber \\
= 8\spbp \varphi+\frac{\sqrt{-1}\alpha}{(2\pi)^2} (-\pbp (\vert \phi \vert_{h_0}^2 e^{-\psi}) +2\tau\pbp \psi_2)  
\end{gather}
Clearly $T$ is a smooth map. Assume that for a given $\alpha \geq 0$ there exists a $\psi, \psi_2, \varphi$ such that $T(\alpha,\psi,\varphi,\psi_2)=0$. We will show that the derivative $DT$ evaluated at this point is a surjective Fredholm operator when acting on $(0,\dot{\psi},  \dot{\varphi}, \dot{\psi_2}) \in 0 \times T\mathcal{B}_{sub} ^{k+2,a}  \times \mathcal{B}_0 ^{k+2,a}$. By the implicit function theorem on Banach manifolds and the Fredholmness of $DT$, this will imply that $T$ is locally onto for an open neighbourhood of $\alpha$.
Indeed, the derivative $DT$ is 
\begin{gather}
DT(0,\dot{\psi}, \dot{\varphi}, \dot{\psi_2})=(S_1, S_2, S_3) \ where \ \nonumber \\
S_1 = \spbp \dot{\psi} + \spbp \dot{\psi_2}-\frac{\vert \phi \vert_h ^2}{4} \dot{\psi} (f+\spbp\phi) + (\frac{\vert \phi \vert_{h_0}^2 e^{-\psi} }{4} -2\lambda) \spbp \dot{\varphi}\nonumber \\
S_2 = \spbp \dot{\psi}_2+\frac{\vert \phi \vert_h ^2}{4} \dot{\psi} (f+\spbp\phi)+(\frac{\tau}{2} -\frac{\vert \phi \vert_{h_0}^2e^{-\psi}}{4}  - 2\lambda)\spbp \dot{\varphi}  \nonumber \\
S_3 = 8 \spbp \dot{\varphi} + \frac{\sqrt{-1}\alpha}{(2\pi)^2} (\pbp (\vert \phi \vert_{h_0}^2 e^{-\psi} \dot{\psi}) +2\tau\pbp \dot{\psi}_2).
\label{derivat}
\end{gather}
We have the following lemma.
\begin{lemma}
The linearisation $DT$ is an elliptic system.
\label{ellipticity}
\end{lemma}
\begin{proof}
The principal symbol of $DT$ is
\begin{gather}
\sqrt{-1} \vert \xi \vert^2 \left [\begin{array}{ccc}
1& 1&(\frac{\vert \phi \vert_{h_0}^2 e^{-\psi} }{4} -2\lambda) \\
0 & 1& (\frac{\tau}{2} -\frac{\vert \phi \vert_{h_0}^2e^{-\psi}}{4}  - 2\lambda) \\
\frac{\alpha}{(2\pi)^2} \vert \phi \vert_{h_0}^2 e^{-\psi}&\frac{2\alpha \tau}{(2\pi)^2} & 8
\end{array} \right ]
\label{princi}
\end{gather}
It is clearly positive-definite if and only if its determinant is positive. The determinant is 
\begin{gather}
8+\frac{2\alpha\tau}{(2\pi)^2} \left [\frac{\vert \phi \vert_h^2}{4} + 2\lambda - \frac{\tau}{2} \right ] + \frac{\alpha}{2(2\pi)^2} \vert \phi \vert _h^2 (\tau-\vert \phi \vert_h^2),\nonumber 
\end{gather}
which is positive-definite because the first term is positive by assumption \ref{satisfaction} and the second by inequality \ref{usefineq}. Therefore the operator is strongly elliptic and is thus Fredholm.
\end{proof}

 By the Fredholm alternative, it is surjective if and only if its formal $L^2$ adjoint $DT^{*}$ from $B_2$ to $B_1$ has a trivial kernel. Indeed the kernel of the formal adjoint consists of functions $u,v,w$ of zero $f$-average such that  
\begin{gather}
\spbp u - \frac{\vert \phi \vert_h ^2 u}{4} (f+\spbp\phi) + \frac{\vert \phi \vert_h ^2 v}{4}(f+\spbp\phi) + \frac{\sqrt{-1}\alpha}{(2\pi)^2}\vert \phi \vert_h ^2 \pbp w  =0 \nonumber \\
\spbp u + \spbp v +\frac{2\alpha\tau}{(2\pi)^2} \spbp w =0 \nonumber \\
\spbp \left [u \left(\frac{\vert \phi \vert_h^2}{4} -2\lambda \right) \right ] + \spbp \left [v \left(\frac{\tau}{2}-\frac{\vert \phi \vert_h^2}{4} -2\lambda \right) \right ]+8\spbp w =0.
\label{adjkern}
\end{gather}
Solving the second and third equations of \ref{adjkern} we get the following equations.
\begin{gather}
\spbp u - \frac{\vert \phi \vert_h ^2 u}{4} (f+\spbp\phi) + \frac{\vert \phi \vert_h ^2 v}{4}(f+\spbp\phi) + \frac{\sqrt{-1}\alpha}{(2\pi)^2}\vert \phi \vert_h ^2 \pbp w  =0 \nonumber \\
u+v=-\frac{2\alpha\tau}{(2\pi)^2}w \nonumber \\
\left [u \left(\frac{\vert \phi \vert_h^2}{4} -2\lambda \right) \right ]+\left [v \left(\frac{\tau}{2}-\frac{\vert \phi \vert_h^2}{4} -2\lambda \right) \right ]+ 8w = constant .
\label{solved}
\end{gather}
Define $q=u-v$. Therefore, 
\begin{gather}
u = \frac{1}{2} \left (q-\frac{2\alpha\tau}{(2\pi)^2}w \right ) \nonumber \\ 
v = -\frac{1}{2} \left ( q+\frac{2\alpha\tau}{(2\pi)^2}w\right ).
\label{quv}
\end{gather}
Writing equations \ref{solved} in terms of $w$ and $q$ we get the following equations.
\begin{gather}
\frac{1}{2}\spbp q -\frac{\vert \phi \vert_h^2}{4} q (f+\spbp\phi) +\frac{\alpha}{(2\pi)^2}(\vert \phi \vert_h^2 -\tau) \spbp w  =0 \nonumber \\
\frac{\vert \phi \vert_h^2 - \tau}{4} q  + w\left(8+ \frac{2\pi c_1(L)\alpha \tau}{(2\pi)^2\mathrm{Vol}(\Sigma)} \right ) =constant.
\label{wandq}
\end{gather}
The following lemma completes the proof.\\

\begin{lemma}
The system of equations \ref{wandq} has a unique smooth solution $q=w=0$. 
\label{completeproof}
\end{lemma}
\begin{proof}
\vspace{-0.1in}
Substituting the second equation of \ref{wandq} in the first, we get the following equation for $q$.
\begin{gather}
\frac{1}{2}\spbp q -\frac{\vert \phi \vert_h^2}{4} q \omega_{\Sigma} -\frac{\alpha (\vert \phi \vert_h^2 -\tau)}{(2\pi)^2 \left(8+ \frac{c_1(L)\alpha \tau}{2\pi\mathrm{Vol}(\Sigma)}\right)} \spbp \left ( \frac{\vert \phi \vert_h^2 - \tau}{4} q\right )  =0.
\label{qeq}
\end{gather}

Multiplying equation \ref{qeq} by $q$ and integrating-by-parts we arrive at
\begin{gather}
-2 \left(8+ \frac{ c_1(L)\alpha \tau}{2\pi\mathrm{Vol}(\Sigma)}\right) \displaystyle \int \sqrt{-1}\partial q \wedge \bar{\partial} q -  \left(8+ \frac{ c_1(L)\alpha \tau}{2\pi\mathrm{Vol}(\Sigma)}\right)\int \vert \phi \vert_h^2 q^2  \omega_{\Sigma}\nonumber \\ +\frac{\alpha}{(2\pi)^2}\int \sqrt{-1}\partial \left( (\vert \phi \vert_h^2 -\tau)q \right) \wedge \bar{\partial} \left ( (\vert \phi \vert_h^2 - \tau) q\right )  =0.
\label{multintqeq}
\end{gather}
Since $\vert a+b \vert ^2 = \frac{\sqrt{-1}(a+b)\wedge (\bar{a}+\bar{b})}{\omega_{\Sigma}} \leq \vert a\vert^2 + \vert b \vert^2 + 2\vert a\vert\vert b \vert$ we see that
\begin{gather}
\sqrt{-1}\partial \left( (\vert \phi \vert_h^2 -\tau)q \right) \wedge \bar{\partial} \left ( (\vert \phi \vert_h^2 - \tau) q\right ) \nonumber \\
\leq (\sqrt{-1}q^2\vert \phi \vert_h^2 \nabla ^{1,0} \phi \wedge \nabla ^{0,1} \phi^{*} + \sqrt{-1}(\vert\phi \vert_h^2 - \tau)^2 \partial q \wedge \bar{\partial} q) + 2\vert a \vert \vert b \vert \omega_{\Sigma} \nonumber \\
= (\sqrt{-1}q^2\vert \phi \vert_h^2 (\pbp \vert \phi \vert_h^2 + \vert \phi \vert_h^2 \Theta_h) + \sqrt{-1}(\vert\phi \vert_h^2 - \tau)^2 \partial q \wedge \bar{\partial} q) + 2 \vert a \vert \vert b \vert \omega_{\Sigma}. 
\label{csq}
\end{gather}
Using equations \ref{CYM} and \ref{uni} we get
\begin{gather}
\sqrt{-1}\partial \left( (\vert \phi \vert_h^2 -\tau)q \right) \wedge \bar{\partial} \left ( (\vert \phi \vert_h^2 - \tau) q\right ) 
\leq \sqrt{-1}(\vert\phi \vert_h^2 - \tau)^2 \partial q \wedge \bar{\partial} q\nonumber \\+ \vert \phi \vert_h^2  q^2 \Bigg(\vert \phi \vert_h^2 \frac{\tau-\vert \phi \vert_h^2}{2}\omega_{\Sigma} + \frac{8(2\pi)^2(\omega_{\Sigma}-f)}{\alpha}+2\tau \left (-\frac{\tau}{4}+\frac{c_1(L)\pi}{\mathrm{Vol}(\Sigma)}+\frac{\vert \phi \vert_h^2}{4}\right)\omega_{\Sigma}\Bigg) + 2 \vert a \vert \vert b \vert \omega_{\Sigma} \nonumber \\
\leq \sqrt{-1}(\vert\phi \vert_h^2 - \tau)^2 \partial q \wedge \bar{\partial} q+ + 2 \vert a \vert \vert b \vert \omega_{\Sigma} \nonumber \\+ \vert \phi \vert_h^2  q^2 \Bigg(\vert \phi \vert_h^2 \frac{\tau-\vert \phi \vert_h^2}{2}\omega_{\Sigma} + \frac{8(2\pi)^2\omega_{\Sigma}}{\alpha}+2\tau \left (-\frac{\tau}{4}+\frac{c_1(L)\pi}{\mathrm{Vol}(\Sigma)}+\frac{\vert \phi \vert_h^2}{4}\right)\omega_{\Sigma}\Bigg).  
\label{csqafter}
\end{gather}
Note that equality holds in the last inequality of \ref{csqafter} if and only if $q=0$. Substituting \ref{csqafter} in \ref{multintqeq} and simplifying we get
\begin{gather}
0\leq \left (-2 \left(8+ \frac{ c_1(L)\alpha \tau}{2\pi\mathrm{Vol}(\Sigma)}\right) + \frac{\alpha(\vert \phi \vert_h^2 -\tau)^2)}{4\pi^2} \right )\displaystyle \int \sqrt{-1}\partial q \wedge \bar{\partial} q +2\frac{\alpha}{(2\pi)^2}\int \vert a \vert \vert b \vert \omega_{\Sigma} \nonumber \\ -  \frac{\alpha}{8\pi^2}\int \vert \phi \vert_h^2 q^2 (\tau-\vert\phi\vert_h^2)^2  \omega_{\Sigma}
\label{aftersubcsqafter}
\end{gather}
Using the Cauchy-Schwartz inequality on the mixed term we get
\begin{gather}
0 \leq \left (-2 \left(8+ \frac{ c_1(L)\alpha \tau}{2\pi\mathrm{Vol}(\Sigma)}\right) + \frac{\alpha(\vert \phi \vert_h^2 -\tau)^2)}{4\pi^2} \right )\displaystyle \int \sqrt{-1}\partial q \wedge \bar{\partial} q  -  \frac{\alpha}{8\pi^2}\int \vert \phi \vert_h^2 q^2 (\tau-\vert\phi\vert_h^2)^2  \omega_{\Sigma} \nonumber \\
+\displaystyle \frac{\alpha}{(2\pi)^2}\int \frac{\sqrt{-1} a \wedge \bar{a}}{g} + \frac{\alpha}{(2\pi)^2}\int g \sqrt{-1} b\wedge \bar{b},  
\label{truecs}
\end{gather}
where we choose the function $g$ to be (again using assumption \ref{satisfaction})
\begin{gather}
g=\frac{\alpha(\tau-\vert \phi \vert_h^2)^2}{8\pi^2} \frac{1}{-\frac{\alpha(\tau-\vert \phi \vert_h^2)^2}{8\pi^2} + 8+\frac{c_1(L)\tau \alpha}{2\pi \mathrm{Vol}(\Sigma)}}.
\end{gather}
Upon substitution of this $g$ into \ref{truecs} and simplifying we get $0\leq 0$ and therefore $q=0$. This means that $w=0$ and hence $u=v=0$. 
\end{proof}
This ends the proof of openness of the set of $\alpha$ (satisfying the assumptions of the theorem) for which the equation has a solution. \qed\\

\emph{Regularity} : In order to complete the proof we need to address the issue of regularity of solutions, i.e., if there is a $C^{2,a}$ solution of \ref{CYM} (where $\omega_{\Sigma}$ is in $C^{0,a}$), is 
it smooth ? (Since this is a semilinear system of equations we need to be a little careful.) Using the second and third equations of \ref{CYM} we can solve for $\omega_{\Sigma}$ as follows.
\begin{gather}
\omega_{\Sigma} \left [4+ \frac{\tau \alpha}{(2\pi)^2} \left(2\lambda-\frac{\tau}{2}+\frac{\vert \phi \vert_h^2}{4}\right) \right ] = 4f + \frac{\sqrt{-1}\alpha}{2(2\pi)^2}\pbp \vert \phi \vert_h^2 .
\label{omegasig}
\end{gather}  
Note that
\begin{gather}
4+ \frac{\tau \alpha}{(2\pi)^2}(2\lambda-\frac{\tau}{2})>0
\end{gather}
by assumption \ref{satisfaction}. Substituting equation \ref{omegasig} in \ref{uni} and using \ref{normalcood} we get
\begin{gather}
\sqrt{-1}\Theta_h + \frac{\vert \phi \vert_h^2 - \tau}{2}\frac{4f + \frac{\sqrt{-1}\alpha}{2(2\pi)^2} ( -\Theta _h \vert \phi \vert_h ^2 + \nabla^{(1,0)} \phi \wedge \nabla ^{(0,1)}\phi^{*})}{4+ \frac{\tau \alpha}{(2\pi)^2} (2\lambda-\frac{\tau}{2}+\frac{\vert \phi \vert_h^2}{4})}=0.
\label{onesys}
\end{gather}
Writing $h=h_0 e^{-\psi}$ we get an equation of the form
\begin{gather}
F_1(e^{-\psi},x)\Delta \psi = F_2(e^{-\psi},\nabla \psi,x)
\label{oftheform}
\end{gather}
where $F_1>0$. Therefore, if $\psi$ is in $C^{2,a}$ then since the right hand side is in $C^{1,a}$ by elliptic regularity $\psi$ is in $C^{3,a}$. By bootstrapping $\psi$ is smooth. This clearly implies that $\omega_{\Sigma}$ and $f_2$ are also smooth. \qed

\end{document}